\documentclass[12pt]{article}

\usepackage{amsmath}
\usepackage{amsfonts}
\usepackage{latexsym}
\usepackage{amssymb}
\usepackage{curves}
\usepackage{graphicx}
\usepackage{amsthm,amscd, amssymb}

\newtheorem{thm}{Theorem}[section]
\newtheorem{pro}[thm]{Proposition}
\newtheorem{cor}[thm]{Corollary}
\newtheorem{lem}[thm]{Lemma}
\newtheorem{conj}[thm]{Conjecture}

\newtheorem{rem}[thm]{Remark}

\newcommand{\mapsfrom}
{\mathrel{\reflectbox{\ensuremath{\mapsto}}}}

\newcommand{\dash}{\mbox{\rm{--}}}

\newcommand{\ua}{\uparrow}

\newcommand{\smid}{\, | \,}

\newcommand{\noin}{\noindent}
\newcommand{\tinyskip}{\vspace{0.05in}}
\newcommand{\dozspace}{\;\;\;\;\;\;\;\;\;\;\;\;}

\newcommand{\FF}{{\mathbb{F}}}

\newcommand{\KK}{{\mathbb{K}}}

\newcommand{\NN}{{\mathbb{N}}}

\newcommand{\Dia}{\mbox{\rm{Dia}}}
\newcommand{\End}{\mbox{\rm{End}}}

\newcommand{\Ind}{\mbox{\rm{Ind}}}

\newcommand{\Iso}{\mbox{\rm{Iso}}}

\newcommand{\Res}{\mbox{\rm{Res}}}

\newcommand{\br}{\mbox{\rm{br}}}

\newcommand{\id}{\mbox{\rm{id}}}
\newcommand{\inc}{\mbox{\rm{inc}}}

\newcommand{\rk}{\mbox{\rm{rk}}}

\newcommand{\tr}{\mbox{\rm{tr}}}

\newcommand{\oo}{\overline}

\renewcommand{\tt}{\widetilde}
\newcommand{\hh}{\widehat}

\newcommand{\cA}{{\cal A}}
\newcommand{\cB}{{\cal B}}
\newcommand{\cC}{{\cal C}}

\newcommand{\cF}{{\cal F}}

\newcommand{\cL}{{\cal L}}
\newcommand{\cM}{{\cal M}}

\newcommand{\cO}{{\cal O}}
\newcommand{\cP}{{\cal P}}

\begin{document}


\title{Pointed fusion systems of blocks}

\author{\large Laurence Barker \\ \mbox{} \\
\normalsize Department of Mathematics \\
\normalsize Bilkent University \\
\normalsize 06800 Bilkent, Ankara \\
\normalsize Turkey}

\maketitle

\small

\begin{abstract}
\noin The pointed fusion system of a block is a
structure consisting of the fusions and relative
multiplicities between the local pointed groups
associated with a maximal Brauer pair. We show
that the pointed fusion system is preserved by
splendid Morita equivalences and part of the
pointed fusion system is preserved by splendid
stable equivalences of Morita type.

\smallskip
\noin 2020 {\it Mathematics Subject Classification.}
Primary: 20C20; Secondary: 20D20.

\smallskip
\noin {\it Keywords:} local pointed group;
Puig category; relative multiplicity; splendid
Morita equivalence; splendid stable equivalence
\end{abstract}

\section{Introduction}
\label{1}

The pointed fusion system of a block,
defined precisely in Section \ref{2}, is
an invariant that can be viewed as a
refinement of the fusion system of a
block. Roughly, the pointed fusion
system is a package of data consisting
of some local pointed groups, the
fusions between the local pointed groups
and the relative multiplicities between the
local pointed groups. The structure of
the pointed fusion system includes the
structure of a category and the structure
of a poset, together with a labelling of
the inclusions by the relative multiplicities.

We mention that a category with at least
some resemblance to the pointed fusion
system was introduced by Puig in \cite{Pui86}
and was named the Puig category in
Th\'{e}venaz \cite[Section 47]{The95}. We
do not know whether the pointed fusion
system, as a category, coincides with the
Puig category. We shall discuss the
relationship between the two categories
in Section \ref{3}.

One of our three main results, Theorem
\ref{3.3}, says that, up to isomorphism, the
pointed fusion system of block can be constructed
from the source algebra of the block. A second
main result, Theorem \ref{5.6}, describes
how a splendid Morita equivalence between
blocks gives rise to an isomorphism of
pointed fusion systems. The third main result,
Theorem \ref{5.7}, describes how a splendid
stable equivalence of Morita type gives rise
to an isomorphism between the stable
parts of the pointed fusion systems.

In view of the first two of those theorems, it
makes sense to view the pointed fusion system
as a local invariant. A consideration of the
$2$-blocks of the non-trivial semidirect
product of $C_3$ over $Q_8$ shows that
the fusion system does not determine
the number of isomorphism classes of
simple modules of a block algebra. As we
shall note in Remark \ref{2.1}, the pointed
fusion system does determine that number.

We shall assume familiarity with block theory,
the theory of Brauer pairs and the theory of
pointed groups, all of which are discussed in
Linckelmann \cite{Lin18}.

After defining, in Section \ref{2}, the
notion of a pointed fusion system and its
stable part, we shall show, in Section \ref{3},
that the pointed fusion system of a block
is determined by a source algebra.

We shall explain, in Section \ref{4}, how
the pointed fusion system of a block yields
a connection between the Weak Donovan
Conjecture and a weak version of Puig's
Conjecture.

In Section \ref{5}, we shall consider two
block algebras of group algebras. We shall
show, in Theorems \ref{5.6} and \ref{5.7},
how two kinds of equivalences between the
two block algebras each give rise to
appropriate isomorphisms between the
structures under consideration.

We shall give some examples. In Section 5,
for blocks with Klein-four defect group,
we shall decribe the three possible
underlying multiposets of the pointed
fusion system. We shall find that there
is only one possibility for the underlying
multiposet of the stable part of the
pointed fusion system.

\section{Pointed fusion systems and their
stable parts}
\label{2}

In the three theorems mentioned in
Section \ref{1}, and also in our pair of
equivalent definitions of the notion of a
pointed fusion system, we shall be making
use of suitable notions of isomorphism.
Formulating those notions of isomorphism
will require some abstraction.

We define a {\bf poset category} to be a
category $\cC$ equipped with:

\tinyskip
\noin $\bullet$ a partial ordering $\leq$
on the set of $\cC$-objects,

\tinyskip
\noin $\bullet$ a family of monomorphisms
${}_P \inc {}_Q : P \leftarrow Q$, called the
{\bf $\cC$-inclusions}, defined for all
$\cC$-objects $P$ and $Q$ satisfying
$P \geq Q$,

\tinyskip
\noin such that the following three
conditions hold:

\tinyskip
\noin {\bf Strictness of
inclusions:} For all $\cC$-objects
$P \geq Q$, the inclusion ${}_P \inc {}_Q$
is an isomorphism if and only if $P = Q$,
furthermore, ${}_P \inc {}_P = \id_P$.

\tinyskip
\noin {\bf Composability of
inclusions:} For all $\cC$-objects
$P \geq Q \geq R$, we have
${}_P \inc {}_Q . {}_Q \inc {}_R =
{}_P \inc {}_R$.

\tinyskip
\noin {\bf Factorization of
morphisms:} For all $\cC$-morphisms
$\phi : P \leftarrow Q$ and $\cC$-objects
$R \leq Q$, there exists a unique
$\cC$-object $\phi(R)$ and a unique
$\cC$-isomorphism $\psi : \phi(R)
\leftarrow R$ such that $\phi . {}_Q \inc
{}_R = {}_P \inc {}_{\phi(R)} . \psi$. We
call $\psi$ the isomorphism with domain
$R$ {\bf restricted} from $\phi$.

\tinyskip
We define a {\bf multiposet} to be
a poset $\cM$ equipped with a function
$m_\cM : \NN \leftarrow \cM \times \cM$,
such that, given $x, y \in \cM$, then
$x \leq y$ if and only if
$m_\cM(x, y) \neq 0$. Thus, a
multiposet is an enrichment of a poset
where each inclusion is labelled with
a positive integer.

Let $\cC$ be a poset category. We
define a {\bf pointed refinement} of
$\cC$ to be a structure with the
following three constituents satisfying
the subsequent four conditions.

\tinyskip
\noin $\bullet$ For each $P \in \cC$,
there is a set $\cP_P$ whose elements
are called the {\bf $\cP$-points} of $P$.
Given $x \in \cP_P$, we write
$P_x = (P, x)$ and we call $P_x$ a
{\bf $\cP$-object}.

\smallskip
\noin $\bullet$ For any $\cC$-morphism
$\phi$ with domain $P$, there is a
bijection $\cP_{\phi(P)} \leftarrow \cP_P$,
written ${}^\phi x \mapsfrom x$, which
depends only on the isomorphism with
domain $P$ restricted from $\phi$. For
$Q \leq P$ and $y \in \cP_Q$, we define
${}^\phi y = {}^\psi y$ where $\psi$ is the
isomorphism with domain $Q$ restricted
from $\phi$. We also write ${}^\phi (Q_y)
= \phi(Q)_{{}^\phi y}$. Note that
${}^\phi (Q_y) = {}^\psi (Q_y)$.

\smallskip
\noin $\bullet$ For any $\cP$-objects
$Q_y$ and $P_x$, there is a natural
number $m_\cP(Q_y, P_x)$ called the
{\bf $\cP$-multiplicity} of $Q_y$ in $P_x$.

\tinyskip
\noin {\bf Bijection composition condition:}
Given $\cC$-morphisms $\psi$ and $\phi$
such that the composite $\psi \phi$ is defined,
letting $P$ be the domain of $\phi$ and
$x \in \cP_P$, then ${}^{\psi \phi} x =
{}^\psi ({}^\phi x)$.

\tinyskip
\noin {\bf Multiposet condition:} The set of
$\cP$-objects, equipped with the function
$m_\cP$, is a multiposet.

\tinyskip
\noin {\bf Refinement condition:} Given
$\cP$-objects $Q_y$ and $P_x$ such that
$Q_y \leq P_x$, then $Q \leq P$,
furthermore, if $Q = P$ then $y = x$

\tinyskip
\noin {\bf Compatibility condition:} Given a
$\cC$-morphism $\phi$ with domain $P$,
a $\cC$-object $Q \leq P$ and points
$x \in \cP_P$ and $y \in \cP_Q$, then
$m({}^\phi (Q_y), {}^\phi (P_x)) =
m(Q_y, P_x)$.

\tinyskip
Let $\cC$ be a poset category and let
$\cP'$ and $\cP$ be pointed refinements
of $\cC$. We define a {\bf $\cC$-identical
isomorphism} $\iota : \cP' \leftarrow \cP$
to be a family of bijections $\iota_P :
\cP'_P \leftarrow \cP_P$ such that $P$
runs over the $\cC$-objects and the
following two conditions hold:

\tinyskip
\noin {\bf Preservation of morphisms:}
For all $\cC$-morphisms $\phi$ with domain
$P$ and $\cP$-points $x$ of $\cP$, we have
$${}^\phi(\iota_P(x)) =
  \iota_{\phi(P)} ({}^\phi x) \; .$$

\noin {\bf Preservation of multiplicities:} For
all $\cC$-objects $Q_y$ and $P_x$, we have
$$m(Q_{\iota_Q(y)}, P_{\iota_P(x)})
  = m(Q_y, P_x) \; .$$

\tinyskip
Throughout the rest of this paper, we let
$\cO$ be a complete local Noetherian ring
with an algebraically closed residue field
$\FF$ of prime characteristic $p$. We let
$G$ be a finite group, we let $b$ a block of
$\cO G$ with defect group $D$ and we let
$B$ a source $D$-algebra of the block
algebra $\cB = \cO G b$. We let $\cF$
be the fusion system on $D$ associated
with $B$. Recall, Linckelmann [Lin18, 8.7.1]
tells us that $\cF$ is determined by the interior
$D$-algebra structure of $B$ and, in fact, by
the interior $\cO D$-$\cO D$-module structure
of $B$.

Given any $p$-subgroup $P$ of $G$, we
write $\br_P : \FF C_G(P) \leftarrow
(\cO G)^P$ for the $P$-relative Brauer
map. Let $\gamma$ be a local point of
$P$ on $\cB$ and let $e$ be a block of
$\FF C_G(P) \br_P(b)$. Consider the
local pointed group $P_\gamma$ and
the Brauer pair $(P, e)$. We write
$P_\gamma \in (P, e)$ when
$\br_P(\gamma) \subset \FF C_G(P) e$.
Given $Q \leq D$, we write $e_Q$ for
the unique block of $\FF C_G(Q)$ such
that $\br_Q(1_B) \in \FF C_G(Q) e_Q$. For
any local pointed group $P_\gamma$ on
$\FF G b$, we say that $P_\gamma$ is
{\bf overshadowed} by $B$ provided
$P \leq D$ and $P_\gamma \in (P, e_P)$.

Given a group $R$ and a monomorphism
$\theta$ with domain $R$, we define
$$\Delta(\theta) =
  \{ (\theta(z), z) : z \in R \}$$
as a subgroup of $\phi(R) {\times} R$.
We define the {\bf pointed fusion
system} $\cL \cP(B)$ of $\cB$ associated
with $B$ to be the pointed refinement of
$\cF$ characterized as follows. Given
$P \leq D$, then the $\cL \cP(B)$-points
of $P$ are those local points $\gamma$
of $P$ on $\cB$ such that $P_\gamma$
is overshadowed by $B$. Thus, the
$\cL \cP(B)$-objects are the pointed
groups on $\cB$ overshadowed by $B$.
Given an $\cL \cP(B)$-object $P_\gamma$
and an $\cF$-morphism $\phi$ with
domain $P$, we define ${}^\phi \gamma$
to be the point of $\phi(P)$ on $\cB$
such that, letting $i \in \gamma$ and
letting $u$ be a unit in $\cB^{\Delta(\phi)}$,
then ${}^u i \in {}^\phi \gamma$. Given
another $\cL \cP(B)$-object $Q_\delta$,
then the $\cL \cP(B)$-multiplicity
$m(Q_\delta, P_\gamma)$ is $0$ unless
$Q \leq P$, in which case
$m(Q_\delta, P_\gamma)$ is the
relative multiplicity of $Q_\delta$ in
$P_\gamma$, we mean, the number of
elements of $\delta$ that appear when
an element of $\gamma$ is expressed
as a sum of mutually orthogonal primitive
idempotents of $A^Q$.

We understand algebras and modules
to be finitely generated over their
coefficient rings. For an algebra $\Lambda$
over $\cO$, we let $\ell(\Lambda)$ denote
the number of isomorphism classes of simple
$\Lambda$-modules.

\begin{rem} \label{2.1}
The number $\ell(\cB)$ is equal to the
number of minimal objects of $\cL \cP(B)$
as a poset.
\end{rem}

\begin{proof}
Since $B$ and $\cB$ are Morita equivalent,
the condition $\tau' \subseteq \tau$
characterizes a bijective correspondence
between the points $\tau'$ on $B$
and the points $\tau$ on the algebra
$C_G(1) e_1 = \cB$. So every pointed group
on $\cB$ having the form $1_\tau$ is
overshadowed by $B$.
\end{proof}

To construct pointed fusion systems
explicitly in particular cases, it is convenient
to work with an isomorphic copy
$\cP(B)$ of $\cL \cP(B)$ defined as
follows. We make use of Brauer
characters of $\FF G$-modules, which we
understand to be $\KK$-valued,
where $\KK$ is a sufficiently large field
of characteristic $0$ whose group of
$p'$-roots of unity is indentified with
the group of torsion units of $\FF$. Some
notation will be needed. Given an algebra
$\Lambda$ over $\cO$ and
$\Lambda$-modules $L$ and $M$ with
$L$ indecomposable, we write $m(L, M)$
to denote the multiplicity of $L$ as a direct
summand of $M$. Given a subalgebra
$\Gamma \leq \Lambda$, we write
${}_\Gamma \Res {}_\Lambda$ to denote
the restriction functor to $\Gamma$-modules
from $\Lambda$-modules.

For $P \leq D$, we define the $\cP(B)$-points
of $P$ to be the irreducible Brauer characters
of $\FF C_G(P) e_P$. Given a $\cP(B)$-point
$\xi$ of $P$ and an $\cF$-morphism
$\phi$ with domain $P$, we define a
$\phi(P)$-point ${}^\phi \xi = {}^g \xi$
where $g \in G$ and $\phi$ is conjugation
by $g$. Let $V(\xi)$ be simple
$\FF C_G(P) e_P$-module with Brauer
character $\xi$. We define the
{\bf diagonal module}
$$\Dia(P_\xi) = \cO G i$$
as an $\FF(G {\times} P)$-module, where
$i$ is a primitive idempotent of $(\cO G)^P$
such that $\br_P(i)$ does not annihilate
$V(\xi)$, and the actions of $G$ and $P$ on
$\cO G i$ are by left and right translation,
respectively. Plainly, $\Dia(P_\xi)$ is
well-defined up to isomorphism, independently
of the choice of $i$. The primitivity of $i$
ensures that $\Dia(P_\xi)$ is indecomposable.
Given $Q \leq P$ and a $\cP(B)$-point $\eta$
of $Q$, we define the $\cP(B)$-multiplicity of
$Q_\eta$ in $P_\xi$ to be
$$m(Q_\eta, P_\xi) = m(\Dia(Q_\eta),
  {}_{G \times Q} \Res {}_{G \times P}
  (\Dia(P_\eta))) \; .$$
Thus, we have defined $\cP(B)$ as a pointed
refinement of $\cF$. The following remark is
clear.

\begin{rem} \label{2.2}
There is an $\cF$-identical isomorphism
$\iota : \cL \cP(B) \leftarrow \cP(B)$ such
that, given a $\cP(B)$-point $P_\xi$, then
$\iota_P(\xi) = \gamma$, where $\gamma$
is the local point of $P$ on $\cB$ such that
$\br_P(\gamma)$ does not annihilate $V(\xi)$.
\end{rem}

Although the pointed fusion system
$\cL \cP(B) \cong \cP(B)$ of $\cB$ does
depend on $B$, the uniqueness of $D$
and $B$ up to $G$-conjugacy implies that,
as a category and as a poset, $\cL \cP(B)$
is well-defined up to isomorphism as an
invariant of $\cB$. To avoid clutter, we
have refrained from writing out the
evident definitions of isomorphism for
poset categories and multiposets. We
have also refrained from writing out the
evident general definition of isomorphism
for pointed refinements of poset categories.
But it is easy to see that, with those
notions of isomorphism understood,
$\cL \cP(B)$ is well-defined up to
isomorphism of pointed refinements
of fusion systems, independently of the
choices of $D$ and $B$.

As an aside, we make some brief
comments about another pointed
refinement of $\cF$ which, again,
can be described, up to
$\cF$-identical isomorphism, in two
ways. Let $\cL \cP^|(B)$ be the
pointed refinement of $\cF$ defined
as follows. For $P \leq D$, we define
the $\cL \cP^|(B)$-points of $P$
to be those points $\alpha$ of $P$
on $\cB$ such that a maximal local
pointed subgroup of $P_\alpha$ is
overshadowed by $B$. Given an
$\cF$-morphism $\phi$ with domain $P$,
we define ${}^\phi \alpha$ as above. The
$\cL \cP^|(B)$-multiplicities are, again, the
usual relative multiplicities between pointed
groups. In analogy with $\cP(B)$, we let
$\cP^|(B)$ be the following pointed
refinement of $\cF$. Employing
some notation and terminology from
\cite{Bar}, we take the $\cP^|(B)$-points
of $P$ to be the substantive generalized
pieces having the form $P {\ua} Q_\eta$
where $Q_\eta$ is a $\cP(B)$-point of
$P$. The action of an $\cF$-morphism
$\phi$ with domain $P$ is given by
${}^\phi(P {\ua} Q_\eta) = ({}^g P)
{\ua} {}^g (Q_\eta)$ where $g$ is as
before. Each $\cP^|(B)$-point $P {\ua}
Q_\eta$ is associated with an
indecomposable $G {\times} P$-bimodule
$\Dia(P {\ua} Q_\eta)$ as defined in
\cite{Bar} and, given a $\cP^|(B)$-point
$S {\ua} T_\tau$ with $S \leq P$, the
$\cP^|(B)$-multiplicity of $S {\ua} T_\tau$
in $P {\ua} Q_\eta$ is defined to be
$$m(S {\ua} T_\tau, P {\ua} Q_\eta) =
  m(\Dia(S {\ua} T_\tau), {}_{G \times S}
  \Res {}_{G \times P}
  (\Dia(P {\ua} Q_\eta))) \; .$$
It is not hard to show that there is an
$\cF$-identical isomorphism
$$\cL \cP^|(B) \cong \cP^|(B) \; .$$
An advantage of $\cL \cP^|(B)$ over
$\cL \cP(B)$, useful when calculating
multiplicties for particular cases, is that,
as explained in \cite{Bar}, the
$\cL \cP^|(B)$-multiplicities
satisfy a matrix relation, and it follows
that all the multiplicties are determined
by those multiplicities $m(S {\ua} T_\tau,
P {\ua} Q_\eta)$ for which $|P : S| = p$.
We omit fuller details because we shall
not be discussing $\cL \cP^|(B)$ any
further in this paper.

We define the stable part of $\cF$,
denoted $\oo{\cF}$, to be the poset
category obtained from $\cF$ by
removing the trivial subgroup of $D$
from the set of objects. Of course, $\cF$
and $\oo{\cF}$ determine each other. The
reason for considering $\oo{\cF}$ is that it
allows us to make the following definition. We
define the {\bf stable part} of $\cL \cP(B)$,
denoted $\oo{\cL \cP}(B)$, to be the pointed
refinement of $\oo{\cF}$ obtained from
$\cL \cP(B)$ by removing the minimal
objects. Thus, the $\oo{\cL \cP}(B)$-objects
are the local pointed groups $P_\gamma$
on $\cB$ such that $1 < P \leq D$ and
$P_\gamma$ is overshadowed by $B$.

\section{Determination by the source algebra}
\label{3}

We continue to work with the block algebra
$\cB = \cO G b$, the source $D$-algebra $B$
and the fusion system $\cF$ introduced in
the previous section. We consider the pointed
fusion system $\cL \cP = \cL \cP (B)$.
Theorem 3.3, below, says that $\cL \cP$
is determined by $B$.

The {\bf Puig category} of $\cB$
associated with $B$, which we write
as $\cL = \cL(B)$, was introduced by
Puig \cite{Pui86}. See also Th\'{e}venaz
\cite[Section 47]{The95}. We define
$\cL$ as follows. The $\cL$-objects are
the local pointed groups on the $D$-algebra
$B$. Given $\cL$-objects $P_\gamma$ and
$Q_\delta$, then the $\cL$-morphisms
$P_\gamma \leftarrow Q_\delta$ are those
group monomorphisms $\phi : P
\leftarrow Q$ such that, choosing
$i \in \gamma$ and $j \in \delta$, then
there exists a unit $u \in B^{\Delta(\phi)}$
satisfying ${}^u j \leq i$, where $\leq$
denotes the usual partial ordering of
idempotents. It is easy to check that
the existence condition is independent
of the choices of $i$ and $j$.

We shall be making use of an
isomorphic copy $\cL'$ of $\cL$ defined
as follows. Let $\lambda$ be the point of
$D$ on $\cB$ such that $1_B \in \lambda$.
In other words, $\lambda$ is the unique
point of $D$ on $B$ such that
$D_\lambda \in (D, e_D)$. The
$\cL'$-objects are the local pointed
subgroups of $D_\lambda$. Given
$\cL'$-objects $P_{\gamma'}$ and
$Q_{\delta'}$, then the $\cL$-morphisms
$P_{\gamma'} \leftarrow Q_{\delta'}$
are the conjugation maps ${}^g y
\mapsfrom y$ where $g \in G$
satisfying $P_{\gamma'} \geq
{}^g (Q_{\delta'})$. A theorem of
Puig \cite[3.6]{Pui86}, also in
Th\'{e}venaz \cite[47.10]{The95},
asserts that there is an isomorphism
of categories $\cL' \leftarrow \cL$
acting as the identity on morphisms
and sending each $\cL$-object
$P_\gamma$ to the $\cL'$-object
$P_{\gamma'}$ such that $\gamma'
\supseteq \gamma$.

Observe that $\cL'$ is a full subcategory
of $\cL \cP$. We do not know whether
$\cL' = \cL \cP$ as categories. We mention
that the conjecture \cite[1.5]{BG22} implies
that $\cL' = \cL \cP$. We do not have an
explicit parameterization of the $\cL'$-objects
or, equivalently, the $\cL$-objects. The next
result, part of Linckelmann \cite[8.7.3]{Lin18},
implies that the inclusion of $\cL'$ in $\cL \cP$
is an equivalence of categories. So we
do have an explicit parameterization of the
isomorphism classes of $\cL$-objects in
terms of irreducible Brauer characters, 
indeed, writing $\cP = \cP(B)$ then,
via the chain of functors $\cP \cong \cL \cP
\hookleftarrow \cL' \cong \cL$, the
isomorphism classes of $\cL$-objects
are in a bijective correspondence with
the $\cP$-objects.

\begin{thm} \label{3.1}
{\rm (Linckelmann.)} Let $P_\gamma$ be a
local pointed group on $\cB$ overshadowed
by $B$. Suppose $P$ is fully $\cF$-centralized.
Then $P_\gamma \leq D_\lambda$, in other
words, $B \cap \gamma$ is a local point of
$P$ on $B$.
\end{thm}

In view of the isomorphism between $\cL$
and the full subcategory $\cL'$ of $\cL \cP$,
the latest theorem has the following
corollary.

\begin{cor} \label{3.2}
Let $P_\gamma$ be a local pointed group
on $B$ such that $P$ is fully $\cF$-centralized.
Let $\phi$ be an $\cF$-automorphism of $P$.
Then there exists a local point ${}^\phi \gamma$
of $P$ on $B$ such that $\phi$ is an
$\cL$-isomorphism $P_{{}^\phi \gamma}
\leftarrow P_\gamma$.
\end{cor}

\begin{thm} \label{3.3}
The pointed fusion system $\cL \cP$ is
determined up to $\cF$-identical
isomorphism by the interior
$D$-algebra structure of $B$.
\end{thm}

\begin{proof}
From $B$, we shall construct a pointed
refinement $\hh{\cP}$ of $\cF$ and an
$\cF$-identical isomorphism $\iota :
\hh{\cP} \leftarrow \cL \cP$. Let $\tt{\cF}$
be a set of fully $\cF$-centralized
subgroups of $D$ such that $\tt{\cF}$
is a set of representatives of the
$\cF$-isomorphism classes. For each
$R \in \tt{\cF}$, let $\cL_R$ denote the
set of points of $R$ on $B$. For each
$P \leq D$, let $\tt{P}$ be the unique
element of $\tt{\cF}$ such that
$P \cong_\cF \tt{P}$. We choose and
fix an $\cF$-isomorphism $\theta_P
: P \leftarrow \tt{P}$ and a set
$\hh{\cP}_P$ together with a bijection
$\Theta_P : \hh{\cP}_P \leftarrow
\cL_{\tt{P}}$. For each $\hh{\gamma}
\in \hh{\cP}_P$, we define
$$\tt{\gamma} =
  \Theta_P^{-1}(\hh{\gamma}) \; .$$
For each $\cF$-morphism $\phi$ with
domain $P$, we define
$$\tt{\phi} = \theta_{\phi(P)}^{-1}
  \phi \theta_P \; .$$
Since $\tt{\phi}$ is an $\cF$-automorphism
of the fully $\cF$-centralized subgroup
$\tt{P}$, Corollary \ref{3.2} ensures that
we can form the local point ${}^{\tt{\phi}}
\tt{\gamma}$ of $\hh{P}$ on $B$. We define
$${}^\phi \hh{\gamma} = \Theta_{\phi(P)}
  ({}^{\tt{\phi}} \tt{\gamma}) \; .$$
Let $Q \leq D$ and $\hh{\delta} \in
\hh{\cP}_Q$. When $Q \not\leq P$, we let
$m_{\hh{\cP}}(Q_{\hh{\delta}},
P_{\hh{\gamma}}) = 0$. Now suppose
$Q \leq P$.  Put $\theta = \theta_P^{-1}
\theta_Q$ as an isomorphism with domain
$\tt{Q}$. We define
$$m_{\hh{\cP}}(Q_{\hh{\delta}},
  P_{\hh{\gamma}}) =
  m({}^\theta (\tt{Q}_{\tt{\delta}}),
  \tt{P}_{\tt{\gamma}})$$
when $\theta$ appears as an $\cL$-isomorphism
with domain $\tt{Q}_{\tt{\delta}}$, otherwise
$m_{\hh{\cP}}(Q_{\hh{\delta}},
P_{\hh{\gamma}}) = 0$. Thus far, we have
specified all the data determining $\hh{\cP}$,
but we have not yet shown that $\hh{\cP}$
is a pointed refinement of $\cF$.

For any $P \leq D$, when $\gamma$
denotes an element of $\cL \cP_P$, it is to
be understood that
$$\tt{\gamma} = B \cap
  {}^{\theta_P^{-1}} \gamma$$
which, by Theorem \ref{3.1}, is a point of
$\tt{P}$ on $B$. For such $\gamma$, it is
also to be understood that
$$\hh{\gamma} = \Theta_P(\tt{\gamma}) \; .$$
We let $\iota$ be the family of bijections
$\iota_P : \hh{P}_P \leftarrow \cL \cP_P$
such that $\iota_P(\gamma) = \hh{\gamma}$.

Simultaneously, we shall show that
$\hh{\cP}$ is a pointed refinement of
$\cF$ and that $\iota$ is an $\cF$-identical
isomorphism $\hh{\cP} \leftarrow \cL \cP$.
We are to show preservation of morphisms
and multiplicities. That is to say, we are to
show that, for all $\cF$-morphisms $\phi$
with domain $P$ and $\gamma \in
\cL \cP_P$, we have ${}^\phi \hh{\gamma}
= \hh{{}^\phi \gamma}$, and we are also
to show that, for all $\cL \cP$-objects
$Q_\delta$ and $P_\gamma$ with
$Q \leq P$, we have $m_{\hh{\cP}}
(Q_{\hh{\delta}}, P_{\hh{\gamma}}) =
m(Q_\delta, P_\gamma)$.

Given $\phi$ and $P_\gamma$ as
specified, then
$$\tt{{}^\phi \gamma} = B \cap
  {}^{\theta_{\phi(P)}^{-1} \phi}
  \gamma = B \cap {}^{\tt{\phi} \,
  \theta_P^{-1}} \gamma =
  {}^{\tt{\phi}} \tt{\gamma} =
  {}^{\tt{\phi}} \Theta_P^{-1}
  (\hh{\gamma}) \; .$$
So $\hh{{}^\phi \gamma} =
\Theta_{\phi(P)}(\tt{{}^\phi \gamma})
= {}^\phi \hh{\gamma}$. We have
established preservation of morphisms.

Given an $\cL \cP$-object $Q_\delta$
with $Q \leq P$, then
$$m(Q_\delta, P_\gamma) =
  m({}^{\theta_P^{-1}} (Q_\delta),
  {}^{\theta_P^{-1}} (P_\gamma)) \; .$$
Let $\theta$ be as above. If $\theta$
appears as an $\cL$-isomorphism with
domain $\tt{Q}_{\tt{\delta}}$, then
$$m({}^{\theta_P^{-1}} (Q_\delta),
  {}^{\theta_P^{-1}} (P_\gamma)) =
  m({}^\theta(\tt{Q}_{\tt{\delta}}),
  \tt{P}_{\tt{\gamma}}) =
  m_{\hh{\cP}}(Q_{\hh{\delta}},
  P_{\hh{\gamma}}) \; .$$
Now suppose $\theta$ does not appear
as an $\cL$-isomorphism with domain
$\tt{Q}_{\tt{\delta}}$. Then
$m_{\hh{\cP}}(Q_{\hh{\delta}},
P_{\hh{\gamma}}) = 0$. Since
$\tt{Q}$ is fully $\cF$-centralized,
Theorem \ref{3.1} implies that
${}^{\theta_Q} (Q_\delta) \leq
D_\lambda$. Hence, by the hypothesis
on $\theta$, we have ${}^{\theta_P^{-1}}
(Q_\delta) \not\leq D_\lambda$. Perforce,
${}^{\theta_P^{-1}} (Q_\delta) \not\leq
{}^{\theta_P^{-1}} (P_\gamma)$. So
$Q_\delta \not\leq P_\gamma$, in other
words, $m(Q_\delta, P_\gamma) = 0$.
We have established preservation of
multiplicities.
\end{proof}

\section{Conjectures on bounds}
\label{4}

Let $G$, $D$, $B$, $\cB$ be as introduced in
Section 2. We continue to work with the pointed
fusion system $\cL \cP = \cL \cP(B)$. We shall
discuss three related conjectures concerning
bounds in terms of the defect group $D$.

A statement of Puig's Conjecture can be
found in Th\'{e}venaz \cite[38.5]{The95}.
Confirmation of an assertion stated without
proof in \cite[38.6]{The95} would imply
that the following conjecture is equivalent
to Puig's Conjecture. For a finitely generated
algebra $\Lambda$ over $\cO$, we define
$\FF \Lambda = \FF \otimes_\cO \Lambda$
as an algebra over $\FF$.

\begin{conj} \label{4.1}
{\rm (Weak Puig Conjecture.)} Fixing $D$,
there is a bound on the dimension of
the source $D$-algebra $\FF B$ of the
block algebra $\FF \cB$.
\end{conj}

The next conjecture was raised in
Eaton--Kessar--K\"{u}lshammer--Sambale
\cite[9.1]{EKKS14}.

\begin{conj} \label{4.2}
{\rm (Weak Donovan Conjecture.)} Fixing
$D$, then there is a bound on the Cartan
invariants of the block algebra $\FF \cB$.
\end{conj}

\begin{conj} \label{4.3}
{\rm (Bounded Multiplicities Conjecture.)}
Fixing $D$, then there is a bound on the
multiplicities of the pointed fusion system
$\cL \cP$ of $\cB$.
\end{conj}

\begin{pro}
\label{4.4}
Fixing $D$, then the Weak Puig Conjecture
holds for $D$ if and only if the Weak
Donovan Conjecture and the Bounded
Multiplicities Conjecture hold for $D$.
\end{pro}

\begin{proof}
Fix $\cB$ and $B$. Let $c$ be the maximum
of the Cartan invariants of $\cB$. Let $m$ be
the maximum of the $\cL \cP$-multiplicities.
We shall show that $c \leq \dim_\FF(\FF B)
\geq m$ and
$$\dim_\FF(\FF B) \leq c m^2 |D|^4 \; .$$

Since $\FF \cB$ and $\FF B$ are Morita
equivalent, they have the same Cartan
invariants and $c \leq \dim_\FF(\FF B)$.
Given $\cL \cP$-objects $Q_\delta$ and
$P_\gamma$ with $Q \leq P$ then, for
all $Q \leq R \leq P$, we have
$$m(Q_\delta, P_\gamma) \geq
  \sum_{\epsilon \in \cL \cP_R}
  m(Q_\delta, R_\epsilon) \,
  m(R_\epsilon, P_\gamma) \; .$$
Therefore, $m = m(1_\tau, D_\lambda)$
for some point $\tau$ on $\cB$. Letting
$W$ be a simple $B$-module not
annihilated by the point $B \cap \tau$
on $B$, then $m = \dim_\FF(W) \leq
\dim_\FF(\FF B)$.

By the above Morita equivalence, the
number $\ell$ of isomorphism classes
of simple $B$-modules is equal to the
number of isomorphism classes of simple
$\cB$-modules. A theorem of Brauer and
Feit in Linckelmann \cite[6.12.1]{Lin18}
asserts that $\ell \leq |D|^2 / 4 + 1$.
If $D$ is trivial, then $\ell = 1$. So
$\ell \leq |D|^2$. We have
$$\dim_\FF (\FF B) = \sum_V
  \dim_\FF(V) \dim_\FF(L_V)$$
where $V$ runs over the simple
$\FF B$-modules up to isomorphism
and $L_V$ denotes the projective
cover of $V$. Each $\dim_\FF(V)
\leq m$ and $\dim_\FF(L_V) \leq
cm \ell$. Therefore, $\dim_\FF(\FF B)
\leq c m^2 \ell^2$.
\end{proof}

In the case where $p = 2$ and $D$ is
abelian, the Weak Donovan Conjecture
was proved in \cite[9.2]{EKKS14}. Hence
we obtain the following corollary.

\begin{cor} \label{4.5}
Fixing $D$ and supposing $p = 2$ and $D$
is abelian, then the Weak Puig Conjecture
holds for $D$ if and only if the Bounded
Multiplicities Conjecture holds for $D$.
\end{cor}

\section{Isomorphisms induced by equivalences}
\label{5}

Let $G$, $b$, $D$, $B$, $\cB$ be as in
Section \ref{2}. In Theorem \ref{5.6}, we
shall describe how a splendid Morita equivalence
from $\cB$ gives rise to an isomorphism of pointed
fusion systems. In Theorem \ref{5.7}, we shall
describe how a splendid stable equivalence of
Morita type from $\cB$ gives rise to isomorphisms
between stable parts of pointed fusion systems.

Given a group $R$, we define $\Delta(R) =
\{ (z, z) : z \in R \}$ as a subgroup of
$R {\times} R$. Recall, given a $p$-subgroup
$P$ of $G$ and an $\cO G$-module $M$, the
{\bf Brauer construction} of $M$ at $P$ is
defined to be the $\FF N_G(P)/P$-module
$$M(P) = M^P /
  \sum_{Q \leq P} \tr_Q^P(M^Q)$$
where $\tr_Q^P$ denotes the transfer map
$M^P \leftarrow M^Q$. A theorem of
Brou\'{e} \cite[3.2]{Bro85} implies that
if $M$ has vertex $P$, then $M(P)$ is
projective and indecomposable.

\begin{lem} \label{5.1}
Let $E$ be a finite $p$-group and let $M$ be
a permutation $\cO E$-$\cO E$-bimodule that
is free as a left $\cO E$-module and as a right
$\cO E$-module. Let $P \leq E$ and let $N$ be
a permutation $\cO E$-$\cO P$-bimodule that
is free as a left $\cO E$-module and as a right
$\cO P$-module. Suppose
$$(M \otimes_{\cO E} N)(\Delta(P)) \neq 0 \; .$$
Then there exists a monomorphism $\phi : E
\leftarrow P$ such that $N(\Delta(\phi)) \neq 0$.
\end{lem}

\begin{proof}
Let $S$ be an $E$-$E$-stable basis for $M$.
Let $T$ be an $E$-$P$-stable basis for $N$.
The hypothesis implies that there exists
$(s, t) \in S {\times} T$ such that
$xs \otimes tx^{-1} = s \otimes t$ for all
$x \in P$. There is a monomorphism $\phi :
E \leftarrow P$ given by $(x s \phi(x)^{-1},
\phi(x) t x^{-1}) = (s, t)$.
\end{proof}

Again, let $P$ be a $p$-subgroup of $G$.
In view of the equality
$$N_{G \times P}(\Delta(P))
  = (C_G(P) {\times} 1) \Delta(P)$$
we have an evident isomorphism
$$N_{G \times P}(\Delta(P))/\Delta(P)
  \cong C_G(P) \; .$$
Via that isomorphism, given an
$\cO(G {\times} P)$-module $M$, we
can regard $M(\Delta(P))$ as an
$\FF C_G(P)$-module.

Adapting a definition in Section \ref{2},
for any pointed group $P_\mu$ on
$\cO G$, choosing $i \in \mu$, we define
the {\bf diagonal module}
$$\Dia(P_\mu) = \cO G i$$
as an $\cO(G {\times} P)$-module.
Again, it is clear that $\Dia(P_\mu)$ is
well-defined independently of the choice
of $i$. Again, the primitivity of $i$ implies
that $\Dia(P_\mu)$ is indecomposable. We
claim that the point $\mu$ of $P$ is
local if and only if $\Dia(P_\mu)$ has
vertex $\Delta(P)$. Since $\Dia(P_\mu)$
is a direct summand of the permutation
$\cO(G {\times} P)$-module $\cO G \cong
\cO(G {\times} P)/\Delta(P)$, some vertex
of $\Dia(P_\mu)$ is contained in $\Delta(P)$.
There is an $\cO$-linear isomorphism
$$\End_{\cO(G \times 1)}(\cO G i)
  \cong \cO G i$$
given by $\sigma \leftrightarrow \sigma(i)$
for an $\cO(G {\times} 1)$-endomorphism
$\sigma$ of $\cO G i$. Given $Q \leq P$, then
$\tr_{\Delta(Q)}^{\Delta(P)} (\sigma) =
\id_{\cO G i}$ if and only if
$\tr_{\Delta(Q)}^{\Delta(P)} (\sigma(i)) = i$.
The claim follows. We mention that diagonal
modules are discussed more systematically in
\cite{Bar}, but our present use of them is
independent of the material there.

Let $P_\gamma$ be a local pointed group on
$\cB$. Then $P_\gamma$ is a pointed group
on $\cO G$ and we can form the diagonal
module $M = \Dia(P_\gamma)$, which is
indecomposable with vertex $\Delta(P)$. Since
$bM = M$, we can regard $M$ as a
$\cB$-$\cP$-bimodule. We have
$\br_P(b) M(\Delta(P)) = M(\Delta(P))$. So
$M(\Delta(P))$ is a projective indecomposable
$\FF C_G(P) \br_P(b)$-module.

\begin{pro} \label{5.2}
Let $P \leq D$. Then the condition
$\Dia(P_\gamma) \cong M$ characterizes a
bijective correspondence $\gamma
\leftrightarrow [M]$ between:

\tinyskip
\noin {\bf (a)} the local points $\gamma$ of
$P$ on $\cB$ such that $P_\gamma$ is
overshadowed by $B$,

\tinyskip
\noin {\bf (b)} the isomorphism classes
$[M]$ of indecomposable
$\cB$-$\cO P$-bimodules $M$ with vertex
$\Delta(P)$ such that $M \smid \cB$ and
$e_P M(\Delta(P)) = M(\Delta(P))$.
\end{pro}

\begin{proof}
By comments above, the specified condition
characterizes a bijective correspondence
between the local points $\gamma$ of $P$
on $\cB$ and the isomorphism classes $[M]$
of indecomposable $\cB$-$\cO P$-bimodules
$M$ with vertex $\Delta(P)$ such that
$M \smid \cB$. For such $\gamma$ and $M$,
supposing $\gamma \leftrightarrow [M]$,
then $P_\gamma$ is overshadowed by $B$ if
and only if $e_P M(\Delta(P)) = M(\Delta(P))$.
\end{proof}

We shall be needing two lemmas describing
how isomorphisms between
$\cL \cP(B)$-objects induce isomorphisms
between diagonal modules and how
multiplicities between $\cL \cP(B)$-objects
can be expressed in terms of diagonal
modules.

\begin{lem} \label{5.3}
Given an $\cL \cP(B)$-object $P_\gamma$
and an $\cF$-morphism $\phi$ with domain
$P$, then
$$\Dia({}^\phi(P_\gamma)) \cong
  \Dia(P_\gamma) \otimes_{\cO P}
  (P {\times} \phi(P))/\Delta(\phi^{-1}) \; .$$
\end{lem}

\begin{proof}
Let $i \in \gamma$ and let $u$ be a unit
in $\cB^{\Delta(\phi)}$. Then ${}^u i \in
{}^\phi \gamma$. We have
$\cO G . {}^u i = \cO G i u^{-1}$ and
$u^{-1} \in \cB^{\Delta(\phi^{-1})}$.
\end{proof}

\begin{lem} \label{5.4}
Given $\cL \cP(B)$-objects $P_\gamma$
and $Q_\delta$ with $Q \leq P$, then
$$m(Q_\delta, P_\gamma) =
  m(\Dia(Q_\delta), {}_{G \times Q}
  \Res {}_{G \times P} (P_\gamma)) =
  m(\Dia(Q_\delta), \Dia(P_\gamma)
  \otimes_{\cO P}
  (P {\times} Q)/\Delta(Q)) \; .$$
\end{lem}

\begin{proof}
The first equality is clear. The functor
$\dash \otimes_{\cO P} (P {\times} Q)
/ \Delta(Q)$ is the restriction functor
to right $\cO Q$-modules from right
$\cO P$-modules.
\end{proof}

We shall also be needing the following part
of Linckelmann \cite[8.7.1]{Lin18}.

\begin{thm} \label{5.5}
{\rm (Linckelmann.)}
Given $P, Q \leq D$, then every
indecomposable direct summand of $B$,
as an $\cO P$-$\cO Q$-bimodule, is
isomorphic to $\cO(P {\times} Q)/
\Delta(\psi)$ for some $\cF$-isomorphism
$\psi$ to a subgroup of $P$ from a
subgroup of $Q$.
\end{thm}

We now introduce another block with the
same defect group $D$. Let $F$ be a finite
group, let $a$ be a block of $\cO F$ with
defect group $D$ and let $A$ be a source
$D$-algebra of the block algebra
$\cA = \cO F a$.

An $\cA$-$\cB$-bimodule $M$ is said to
induce a {\bf splendid Morita equivalence}
to $\cA$ from $\cB$ with respect to $A$ and
$B$ provided the following two conditions hold:

\tinyskip
\noin $\bullet$ $M$ and the dual $M^*$
induce a Morita equivalence between $\cA$
and $\cB$,

\tinyskip
\noin $\bullet$ $M$ is an indecomposable
direct summand of $\cO F 1_A
\otimes_{\cO D} 1_B \cO G$.

\tinyskip
\noin A theorem of Puig and Scott in
\cite[9.7.4]{Lin18} asserts that there is a
splendid Morita equivalence to $\cA$ from
$\cB$ with respect to $A$ and $B$ if and only
if there is an interior $D$-algebra isomorphism
$A \cong B$. We mention that the hypothesis
on the coefficient ring in \cite{Lin18} is slightly
different, but the proof in \cite{Lin18} carries
over, without change, to the case of arbitrary
$\cO$. Note that, when the two equivalent
conditions hold, the fusion system $\cF$
associated with $B$ is also the fusion system
associated with $A$.

\begin{thm} \label{5.6}
Suppose there is
an $\cA$-$\cB$-bimodule $M$ inducing a
splendid Morita equivalence to $\cA$ from
$\cB$ with respect to $A$ and $B$. Then
an $\cF$-identical isomorphism $\iota :
\cL \cP(A) \leftarrow \cL \cP(B)$ is given
by $\Dia(P_{\iota_P(\gamma)}) \cong
M \otimes_\cB \Dia(P_\gamma)$
for any $\cL \cP(B)$-object $P_\gamma$.
\end{thm}

\begin{proof}
Fix an $\cL \cP(B)$-object $P_\gamma$
and let
$$L = M \otimes_\cB \Dia(P_\gamma)$$
as an $\cA$-$\cO P$-bimodule. We shall
show that there exists a point $\alpha$ of
$P$ on $\cA$ such that $P_\alpha$ is a
$\cL \cP(A)$-object and $L \cong
\Dia(P_\alpha)$. Since $\Dia(P_\gamma)
\cong M^* \otimes_\cA L$, it will then follow,
by Proposition \ref{5.2}, that the function
$P_\alpha \mapsfrom P_\gamma$ is a
bijection to the set of $\cL \cP(A)$-objects
from the set of $\cL \cP(B)$-objects.

The functor $M \otimes_\cB \dash$ is a
Morita equivalence to $\cA \otimes_\cO
\cO P$ from $\cB \otimes_\cO \cO P$, so
$L$ is indecomposable. Letting
$i \in \gamma$, we have $\Dia(P_\gamma)
\cong \cO G i$, hence $L \cong M i$. So
$L \smid \cO F 1_A \otimes_{\cO D} W$
for some indecomposable direct summand
$W$ of the $\cO D$-$\cO P$-bimodule
$1_B \cO G i$. Therefore,
$$\Dia(P_\gamma) \cong
  M^* \otimes_\cA L \smid
  M^* 1_A \otimes_{\cO D} W \; .$$
Since $\Dia(P_\gamma)$ has vertex
$\Delta(P)$, we have $\Dia(P_\gamma)
(\Delta(P)) \neq 0$, so $(M^* 1_A
\otimes_{\cO D} W) (\Delta(P)) \neq 0$.
By Lemma \ref{5.1}, $W(\Delta(\psi))
\neq 0$ for some monomorphism
$\psi : D \leftarrow P$. Since $W$ is
indecomposable, $W \cong \cO(D
{\times} P)/\Delta(\psi)$. By Theorem
\ref{5.5}, $\psi$ is an $\cF$-morphism.
Let $I = \cO(\psi(P) {\times} P)/\Delta(\psi)$
as an $\cO \psi(P)$-$\cO P$-bimodule.
Writing $I^\circ$ for the opposite
bimodule of $I$, we have
$$W \otimes_{\cO P} I^\circ \cong
  \cO(D {\times} \psi(P)) /
  \Delta(\psi(P)) \; .$$
So $\dash \otimes_{\cO D} W
\otimes_{\cO P} I^\circ$ is the restriction
functor to right $\cO \psi(P)$-modules
from right $\cO D$-modules. Therefore,
$$L \otimes_{\cO P} I^\circ \smid
  \cO F 1_A \otimes_{\cO D} W
  \otimes_{\cO P} I^\circ \cong \cO F 1_A$$
as $\cA$-$\cO \psi(P)$-bimodules. Since
$L \otimes_{\cO P} I^\circ$ is
indecomposable, there exists a
primitive idempotent $h'$ of
$A^{\psi(P)}$ such that
$$L \otimes_{\cO P} I^\circ
  \cong \cO F h'$$
as $\cA$-$\cO \psi(P)$-bimodules. Let
$f \in F$ such that $\psi$ is conjugation
by $f$. Let $h$ be the primitive
idempotent of $A^P$ such that
$h' = {}^f h$. Let $\alpha$ be the point
of $P$ on $\cA$ owning $h$. We have
$$L \cong L \otimes_{\cO P} I^\circ
  \otimes_{\cO \psi(P)} I \cong
  \cO F h' \otimes_{\cO \psi(P)}
  \cO \psi(P) f P \cong \cO F h
  \cong \Dia(P_\alpha) \; .$$

To show that $\alpha$ is local, let
$Q$ be a defect group of $\alpha$.
As an $\cO(F {\times} P)$-module,
$L$ is isomorphic to a direct summand
of a module induced from $F {\times} Q$.
By Green's indecomposability Criterion,
$$L \cong {}_{F {\times} P} \Ind
  {}_{F \times Q}(K) \cong
  K \otimes_{\cO Q} \cO P$$
for some $\cA$-$\cO Q$-bimodule $K$.
We have
$$\Dia(P_\gamma) \cong M^* \otimes_\cA L
  \cong M^* \otimes_\cA K \otimes_{\cO Q}
  \cO P \cong {}_{G \times P} \Ind
  {}_{G \times Q} (M^* \otimes_\cA K) \; .$$
But $\Dia(P_\gamma)$ has vertex
$\Delta(P)$, so $Q = P$ and $\alpha$ is
local.

To show that $P_\alpha$ is overshadowed
by $A$, let $\alpha'$ be the point of
$\psi(P)$ on $\cA$ owning $h'$. Since
$h' \in A^{\psi(P)}$, we have
$\psi(P)_{\alpha'} \in (\psi(P), e_{\psi(P)})$.
But $\psi(P)_{\alpha'} = {}^f (P_\alpha)$
and $(\psi(P), e_{\psi(P)}) = {}^f (P, e_P)$,
so $P_\alpha \in (P, e_P)$, in other words,
$P_\alpha$ is overshadowed by $A$.

We have now established that $P_\alpha$
is an $\cL \cP(A)$-object. It remains to show
that $\iota_P$ preserves fusions and
multiplicities. Let $\phi$ be an
$\cF$-morphism with domain $P$. Lemma
\ref{5.3} tells us that, applying the functor
$\dash \otimes_{\cO P} (P {\times} \phi(P))
/ \Delta(\phi^{-1})$, we have
$$\Dia({}^\phi (P_\alpha)) \cong
  M \otimes \Dia({}^\phi(P_\gamma)) \; .$$
So ${}^\phi \alpha = \iota_{\phi(P)}
({}^\phi \gamma)$ and $\iota$ preserves
fusions. Let $Q \leq P$ and let $\delta$
be an $\cL \cP(B)$-point of $Q$. Let $\beta$
be the $\cL \cP(A)$-point of $Q$ such that
$\Dia(Q_\beta) \cong M \otimes_\cB
\Dia(Q_\delta)$. Since $M \otimes_\cB \dash$
is a Morita equivalence to $\cA \otimes_\cO
\cO Q$ from $\cB \otimes_\cO \cO Q$,
Lemma \ref{5.4} yields $m(Q_\beta, P_\alpha)
= m(Q_\delta, P_\gamma)$. So $\iota$
preserves multiplicities.
\end{proof}

An $\cA$-$\cB$-bimodule $M$ is said to
induce a {\bf splendid stable equivalence
of Morita type} to $\cA$ from $\cB$ with
respect to $A$ and $B$ provided the
following two conditions hold:

\tinyskip
\noin $\bullet$ $M$ and the dual $M^*$
induce a stable equivalence of Morita type
between $\cA$ and $\cB$,

\tinyskip
\noin $\bullet$ $M$ is an indecomposable
direct summand of $\cO F 1_A
\otimes_{\cO D} 1_B \cO G$.

\tinyskip
\noin By \cite[9.8.2]{Lin18}, when such
$M$ exists, the fusion system associated
with $A$ is $\cF$.

\begin{thm} \label{5.7}
Suppose there is an $\cA$-$\cB$-bimodule
$M$ inducing a splendid stable equivalence
of Morita type to $\cA$ from $\cB$ with respect
to $A$ and $B$. Then there is an $\cF$-identical
isomorphism $\oo{\iota} : \oo{\cL \cP}(A)
\leftarrow \oo{\cL \cP}(B)$ such that
$\Dia(P_{\oo{\iota}_P(\gamma)})$ is the
non-projective part of $M \otimes_\cB
\Dia(P_\gamma)$ for any
$\oo{\cL \cP}(B)$-object $P_\gamma$.
\end{thm}

\begin{proof}
Write $M^* \otimes_\cA M
\cong \cB \oplus Y$ where $Y$ is a
projective $\cB$-$\cB$-module. Fix an
$\oo{\cL \cP}(B)$-object $P_\gamma$.
We have a direct sum of
$\cA$-$\cO P$-bimodules
$$L \oplus L' \cong M \otimes_\cB
  \Dia(P_\gamma)$$
where $L$ is indecomposable and
non-projective while $L'$ is projective.
Letting $i \in \gamma$, then
$\Dia(P_\gamma) \cong \cO G i$ and
$$M^* \otimes_\cA L \oplus M^*
  \otimes_\cA L' \cong (\cB \oplus Y)
  \otimes_\cB \Dia(P_\gamma) \cong
  \Dia(P_\gamma) \oplus Yi \; .$$
Since $M^* \otimes_\cA L'$ and $Yi$ are
projective, $\Dia(P_\gamma)$ is the
non-projective part of $M^* \otimes_\cA L$.
We shall show that there exists an
$\oo{\cL \cP}(A)$-object $P_\alpha$ such
that $L \cong \Dia(P_\alpha)$. It will then
follow, by Proposition \ref{5.2}, that there
is a bijective correspondence $P_\alpha
\leftrightarrow P_\gamma$ between
the $\oo{\cL \cP}(A)$-objects and the
$\oo{\cL \cP}(B)$-objects.

We have $L \oplus L' \cong Mi$. So
$L \smid \cO F 1_A \otimes_{\cO D} W$
for some indecomposable
$\cO D$-$\cO P$-bimodule $W$ such
that $W \smid 1_B \cO G i$. Since
$\Dia(P_\gamma)$ is the
non-projective part of
$M^* \otimes_\cA L$, we have
$$(M^* \otimes_\cA L)(\Delta(P))
  \cong (\Dia(P_\gamma))(\Delta(P))
  \neq 0 \; .$$
But $M^* \otimes_\cA L \smid
M^* \otimes_\cA \cO F 1_A
\otimes_{\cO D} W \cong
M^* 1_A \otimes_{\cO D} W$. So
$$(M^* 1_A \otimes_{\cO D} W)
  (\Delta(P)) \neq 0 \; .$$

The next stage of the argument proceeds
much as in the proof of Theorem \ref{5.6}.
Let us summarize it. Again, we find that
$W \cong \cO(D {\times} P)/\Delta(\psi)$
for some $\cF$-morphism $\psi$. Letting
$I$ be as before, we find that there exists a
primitive idempotent $h'$ of $A^{\psi(P)}$
such that $L \otimes_{\cO P} I^\circ
\cong \cO F h'$. By considering $f$,
$h$, $\alpha$ as before, we deduce that
$L \cong \Dia(P_\alpha)$.

To show that $\alpha$ is local, adaptation
of the analogous argument in the proof
of Theorem \ref{5.6} requires some care.
We again use Green's Indecomposability
Criterion to show that $L \cong K
\otimes_{\cO Q} \cO P$ where $K$ is
an $\cA$-$\cO Q$-bimodule and $Q$
is a defect group of $\alpha$. Again,
$$M^* \otimes_\cA L \cong
  {}_{G \times P} \Ind {}_{G \times Q}
  (M^* \otimes_\cA K) \; .$$
But we saw above that
$(M^* \otimes_\cA L)(\Delta(P)) \neq 0$.
So $Q = P$ and $\alpha$ is local. To show
that $P_\alpha$ is overshadowed by $A$
the argument is very similar to what we
did before.

Thus, we have established that $P_\alpha$
is an $\oo{\cL \cP}(A)$-object, and it
remains only to check preservation of
fusions and multiplicities. Let $\phi$ be
as before. Write $J = \cO(P {\times}
\phi(P))/\Delta(\phi^{-1})$. Applying the
functor $\dash \otimes_{\cO P} J$ to the
isomorphism $\Dia(P_\alpha) \oplus L'
\cong M \otimes_\cB \Dia(P_\gamma)$
and using Lemma \ref{5.3}, we obtain
$$\Dia({}^\phi(P_\alpha)) \oplus L'
  \otimes_{\cO P} J \cong M \otimes_\cB
  \Dia({}^\phi (P_\gamma)) \; .$$
Since $L' \otimes_{\cO P} J$ is a
projective $\cA$-$\cO \phi(P)$-bimodule,
${}^\phi \alpha = \oo{\iota}_{\phi(P)}
({}^\phi \gamma)$ and $\oo{\iota}$
preserves fusions.

To show that $\oo{\iota}$ preserves
multiplicities, let $Q$ be a non-trivial
subgroup of $P$, let $\delta$ be an
$\oo{\cL \cP}(B)$-point of $Q$ and let
$\beta$ be the $\oo{\cL \cP}(A)$-point
of $Q$ such that
$$\Dia(Q_\beta) \oplus N \cong
  M \otimes_\cB \Dia(Q_\delta)$$
for some projective
$\cB$-$\cO Q$-bimodule $N$. Since
$L' \otimes_{\cO P} \cO(P {\times} Q)
/ \Delta(Q)$ is a projective
$\cA$-$\cO Q$-bimodule, Lemma
\ref{5.4} yields
$$m(Q_\beta, P_\alpha) =
  m(\Dia(Q_\beta), M \otimes_\cB
  \Dia(P_\alpha) \otimes_{\cO P}
  \cO(P {\times} Q)/\Delta(Q)) \; .$$
The non-projective part of
$M^* \otimes_\cA \Dia(Q_\beta)$ is
$\Dia(Q_\delta)$ and the non-projective
part of $M^* \otimes_\cA M \otimes_\cB
\Dia(P_\gamma) \otimes_{\cO P}
\cO(P {\times} Q)/\Delta(Q)$ is
isomorphic to the non-projective part
of $\Dia(P_\gamma) \otimes_{\cO P}
(P {\times} Q)/\Delta(Q)$, so
$$m(Q_\beta, P_\alpha) =
  m(\Dia(Q_\delta), \Dia(P_\gamma)
  \otimes_{\cO P} \cO (P {\times} Q)
  /\Delta(Q)) \; .$$
By Lemma \ref{5.4} again, $m(Q_\beta,
P_\alpha) = m(Q_\delta, P_\gamma)$.
So $\oo{\iota}$ preserves
multiplicities.
\end{proof}

\section{Klein-four defect groups}
\label{6}

Once again, let $G$, $b$, $D$, $B$, $\cB$
be as in Section 2. In the case where $p = 2$
and $D \cong V_4$, we shall describe all the
possibilities for the underlying multiposet of the
pointed fusion system $\cL \cP = \cL \cP(B)$. This
will be an application of the following theorem of
Craven--Eaton--Kessar--Linckelmann
\cite[1.1]{CEKL11}. Their proof of the theorem
relies on the classification of simple finite groups.

\begin{thm} \label{6.1}
{\rm (Craven--Eaton--Kessar--Linckelmann.)}
{\it Supposing $p = 2$ and $D \cong V_4$
then, as an interior $D$-algebra, $B$ is
isomorphic to $\cO D$ or $\cO A_4$ or the
principal block algebra of $\cO A_5$. In the
latter two cases, $D$ is identified with a
Sylow $2$-subgroup of $A_4$ or $A_5$.}
\end{thm}

Suppose $p = 2$ and $D \cong V_4$. By
Theorem \ref{6.1}, together with Theorem
\ref{5.6}, we shall have described all the
possible multiposet structures for $\cL \cP$
when we have done so in the three cases
where $G \in \{ D, A_4, A_5 \}$ and $b$ is
the principal block of $\cO G$.

Let $X$, $Y$, $Z$ be the proper subgroups
of $D$. For any $R \in \{ X, Y, Z, D \}$, we
have $C_G(R) = D$. So there exists a unique
local point $\gamma^R$ of $R$ on $\cB$. We
write $R_1 = R_{\gamma^R}$. Let $V_1$,
$...$ be representatives of the isomorphism
classes of simple $\cB$-modules, enumerated
such that $V_1$ is trivial. Let $\gamma_i$ be
the point on $\cB$ that does not annihilate
$V_i$. We write $1_i = 1_{\gamma_i}$.
Thus, every $\cL \cP$-object has the form
$R_1$ or $1_i$. We shall show that, in the
three cases where $G$ is $D$, $A_4$, $A_5$,
respectively, the multiposet structure of
$\cL \cP$ is as shown, where the double
lines indicate multiplicity $2$ and all the
other multiplicities are $1$.

\smallskip
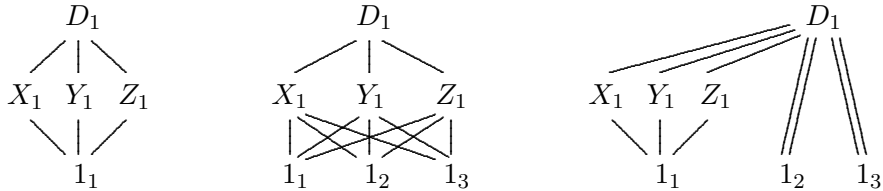
\hspace{0.4in}
\begin{picture}(400,76)


\put(33,7){$1_1$}
\put(8,37){$X_1$}
\put(30,37){$Y_1$}
\put(50,37){$Z_1$}
\put(30,67){$D_1$}

\curve(30,18,17,31)
\curve(35,18,35,31)
\curve(40,18,53,31)

\curve(17,49,30,61)
\curve(35,49,35,61)
\curve(53,49,40,61)


\put(112,7){$1_1$}
\put(143,7){$1_2$}
\put(173,7){$1_3$}
\put(108,37){$X_1$}
\put(140,37){$Y_1$}
\put(170,37){$Z_1$}
\put(140,67){$D_1$}

\curve(115,18,115,31)
\curve(118,17,141,32)
\curve(121,16,170,33)

\curve(140,17,118,32)
\curve(145,18,145,31)
\curve(150,17,173,32)

\curve(170,16,121,33)
\curve(175,17,149,32)
\curve(176,18,176,31)

\curve(117,49,140,61)
\curve(145,49,145,61)
\curve(173,49,150,61)


\put(253,7){$1_1$}
\put(228,37){$X_1$}
\put(250,37){$Y_1$}
\put(270,37){$Z_1$}
\put(310,67){$D_1$}
\put(300,7){$1_2$}
\put(329,7){$1_3$}

\curve(250,18,237,31)
\curve(255,18,255,31)
\curve(260,18,273,31)

\curve(236,49,304,67)
\curve(255,49,306,65)
\curve(273,49,308,63)

\curve(301,18,311,62)
\curve(304,18,314,62)

\curve(330,18,320,62)
\curve(333,18,323,62)

\end{picture}

For the rest of this paper, we regard
$\cL \cP$ as a multiposet. Plainly, if
$G = D$, then $\cL \cP$ is as depicted
in the left-hand diagram. The remaining
two cases share some common features.
Henceforth, suppose that $G = A_4$ or
$G = A_5$ and let $b$ be the principal
block of $\cO G$. By Theorem \ref{6.1}
(or an easy direct argument which we omit),
$B = \cB$. The points on $\cB$ are the
$\gamma_i$ with $i \in \{ 1, 2, 3 \}$. Let
$E_i$ be the indecomposable projective
$\cB$-module, well-defined up to
isomorphism, such that $V_i \cong
E_i / J(E_i)$. Transporting via the
isomorphism $G {\times} 1 \cong G$,
we have
$$\Dia(1_i) \cong {}_{G \times 1}
  \Iso {}_G (E_i) \; .$$
To proceed further, we consider the
two cases separately.

Suppose $G = A_4$. Then $\cO G \cong
E_1 \oplus E_2 \oplus E_3$ as
$\cO G$-modules. So
$$\cO G \cong \Dia(1_1) \oplus
  \Dia(1_2) \oplus \Dia(1_3)$$
as $\cO(G {\times} 1)$-modules. We also
have $\cO G \cong \Dia(X_1) \oplus E$ as
$\cO(G {\times} X)$-modules, where $E$
is projective. Restricting via the
embedding $G \cong G {\times} 1
\hookrightarrow G {\times} X$, we have
${}_G \Res {}_{G {\times} X}(E) \smid E_1
\oplus E_2 \oplus E_3$. But every projective
$\cO(G {\times} X)$-module restricts to a
direct sum of $2$ isomorphic copies of a
projective $\cO G$-module. Therefore,
$E = 0$, that is, $\Dia(X_1) \cong \cO G$
as $\cO(G {\times} X)$-modules and
$${}_{G \times 1} \Res {}_{G \times X}
  (\Dia(X_1)) \cong \Dia(1_1) \oplus
  \Dia(1_2) \oplus \Dia(1_3) \; .$$
It follows that $\Dia(D_1) \cong \cO G$
as $\cO(G {\times} D)$-modules and
$${}_{G \times X} \Res {}_{G \times D}
  (\Dia(D_1)) \cong \Dia(X_1) \; .$$
Bearing in mind that the subgroups
$X$, $Y$, $Z$ are $G$-conjugate, we
deduce, using Lemma \ref{5.4}, that
$\cL \cP$ is as depicted in the
middle diagram above.

Now suppose $G = A_5$. Let
$X < H < G$ with $H \cong D_{10}$, the
dihedral group of order $10$. We have
$\cO H \cong L \oplus L_0$ as
$\cO(H {\times} X)$-modules, where
$L_0$ is projective and $L$ is
indecomposable with vertex $\Delta(X)$
and $\cO$-rank $\rk_\cO(L) = 2$.
Since $\Dia(X_1)$ has vertex
$\Delta(X)$ and
$$\Dia(X_1) \smid \cO G \cong
  {}_{G {\times} X} \Ind
  {}_{H {\times} X} (\cO H)$$
we have $\Dia(X_1) \smid {}_{G \times X}
\Ind {}_{H {\times} X} (L)$. Therefore,
$\rk_\cO(\Dia(X_1)) \leq 12$. For
$i \in \{ 1, 2, 3 \}$, inducing via the
embedding $G {\times } X \hookleftarrow
G {\times} 1 \cong G$, let $E_i^X =
{}_{G \times X} \Ind {}_G (E_i)$.
Then $E_1^X$, $E_2^X$, $E_3^X$
comprise a set of representatives of
the isomorphism classes of
indecomposable projective
$\cO(G {\times} X)$-modules. Since
$\rk_\cO(E_1) = 12$ and
$\rk_\cO(E_2) = \rk_\cO(E_3) = 8$,
we have $\rk_\cO(E_1^X) = 24$ and
$\rk_\cO(E_2^X) = \rk_\cO(E_3^X)
= 16$. Now
$$\cB \cong \Dia(X_1) \oplus E$$
as $\cO(G {\times} X)$-modules, where
$E$ is projective. We have
$\rk_\cO(\cB) = 44$, so
$32 \leq \rk_\cO(E) < 44$. By
considering an outer automorphism of
$G$, we see that $E_2^X$ and
$E_3^X$ have the same multiplicity as
direct summands of $E$. The constraints
we have obtained on the $\cO$-ranks
imply that $E \cong E_2^X \oplus E_3^X$.
Therefore, $\rk_\cO(\Dia(X_1)) = 12$.

We have $\cB \cong E_1 \oplus 2 E_2
\oplus 2 E_3$ as $\cO G$-modules, so
$$\cB \cong \Dia(1_1) \oplus
  2 \Dia(1_2) \oplus 2 \Dia(1_3)$$
as $\cO(G {\times} 1)$-modules. By
the above isomorphism for $E$, we have
$${}_{G \times 1} \Res {}_{G \times X}
  (\Dia(X_1)) \cong \Dia(1_1) \; , \dozspace
  {}_{G \times 1} \Res {}_{G \times X} (E)
  \cong 2 \Dia(1_2) \oplus 2 \Dia(1_3) \; .$$
We have $\Dia(D_1) \cong \cO G 1_B$
as $\cO(G {\times} D)$-modules. But
$B = \cB$, so $\Dia(D_1) \cong \cB$ as
$\cO(G {\times} D)$-modules and
$${}_{G \times X} \Res {}_{G \times D}
  (\Dia(D_1)) \cong \Dia(X_1)
  \oplus E \; .$$
Again bearing in mind the $G$-conjugacy
of $X$, $Y$, $Z$, an application of
Lemma \ref{5.4} yields the conclusion
that $\cL \cP$ is as depicted in the
right-hand diagram above.

Our analysis of the three cases is now
complete. It follows, in particular, that
whenever the defect group of a $2$-block
is $V_4$, the underlying multiposet of
the stable part $\oo{\cL \cP}$ of the
pointed fusion system is such that all the
$\oo{\cL \cP}$-multiplicities are unity
and, as a poset, $\oo{\cL \cP}$ has the
following Hasse diagram.

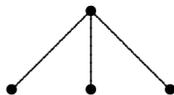
\hspace{0.7in}
\begin{picture}(300,43)

\put(120, 7){\circle*{4}}
\put(150, 7){\circle*{4}}
\put(180, 7){\circle*{4}}
\put(150, 37){\circle*{4}}

\curve(120,7,150,37)
\curve(150,7,150,37)
\curve(180,7,150,37)

\end{picture}


\begin{thebibliography}{EMG}

\bibitem[1]{Bar}
L.\ Barker, {\it The pointed groups on a block
algebra}, (preprint).

\bibitem[2]{BG22}
L.\ Barker, M.\ Gelvin, {\it Conjectural
invariance with respect to the fusion system
of an almost-source algebra}, J.\ Group
Theory {\bf 25}, 973-995 (2022).

\bibitem[3]{Bro85}
M.\ Brou\'{e}, {\it On Scott modules and
$p$-permutation modules: an approach through
the Brauer morphism}, Proc.\  American Math.\
Soc.\ {\bf 93}, 401-408 (1985).

\bibitem[4]{CEKL11}
D.\ A.\ Craven, C.\ W.\ Eaton, R.\ Kessar,
M.\ Linckelmann, {\it The structure of blocks with
Klein four defect group}, Math.\ Zeit.\ {\bf 208},
441-476 (2011).

\bibitem[5]{EKKS14}
C.\ W.\ Eaton, R.\ Kessar, B.\ K\"{u}lshammer,
B.\ Sambale, {\it $2$-Blocks with abelian defect
groups}, Advances in Mathematics {\bf 254},
706-735 (2014).

\bibitem[6]{Lin18}
M.\ Linckelmann, ``The Block Theory of Finite Group
Algebras'', Vols. 1, 2 (Cambridge University Press,
Cambridge, 2018).

\bibitem[7]{Pui86}
L.\ Puig, {\it Local fusions in block source algebras},
J.\ Algebra {\bf 104}, 358-369 (1986).

\bibitem[8]{The95}
J.\ Th\'{e}venaz, ``$G$-Algebras and Modular
Representation Theory'', (Clarendon Press,
Oxford, 1995).

\end{thebibliography}
\end{document}